\documentclass[11pt]{amsart}
\usepackage{amssymb}
\usepackage{dsfont}

\newtheorem{theorem}{Theorem}

\newcommand{\B}{\mathcal B}

\newcommand{\R}{\mathbb R}

\newcommand{\ce}{\mathfrak c}

\author{Marek Balcerzak}
\address{Institute of Mathematics, Lodz University of Technology, W\'olcza\'nska 215, 93-005
\L\'od\'z, Poland}
\email{marek.balcerzak@p.lodz.pl}
\author{Szymon G\l \c ab}
\address{Institute of Mathematics, Lodz University of Technology, W\'olcza\'nska 215, 93-005
\L\'od\'z, Poland}
\email {szymon.glab@p.lodz.pl}

\title{A lower density operator for the Borel algebra}
\date{}

\begin{document}

\begin{abstract}
We prove that the existence of a Borel lower density operator (a Borel lifting) with respect to the $\sigma$-ideal
of countable sets, for an uncountable Polish space, is equivalent to the Continuum Hypothesis.
\end{abstract}
\maketitle

Let $S$ be a $\sigma$-algebra of subsets of a nonempty set $X$ and let $J\subseteq S$ be a $\sigma$-ideal.
We write $A\sim B$ whenever $A\bigtriangleup B\in J$.
A mapping
$\Phi\colon S\to S$ is called a {\em lower density operator} (respectively, a {\em lifting}) with respect to $J$ if it satisfies the following conditions
(1)--(4) (respectively, (1)--(5)):
\begin{itemize}
\item[(1)] $\Phi(X)=X$ and $\Phi(\emptyset)=\emptyset$,
\item[(2)] $A\sim B\implies\Phi(A)=\Phi(B)$ for every $A,B\in S$,
\item[(3)] $A\sim\Phi(A)$ for every $A\in S$,
\item[(4)] $\Phi(A\cap B)=\Phi(A)\cap\Phi(B)$ for every $A,B\in S$,
\item[(5)] $\Phi(A\cup B)=\Phi(A)\cup\Phi(B)$ for every $A,B\in S$.
\end{itemize}

The problem of the existence of liftings together with their various applications were widely discussed in the monograph \cite{IT} and in the later survey \cite{SMM}. If $S$ is the $\sigma$-algebra of Borel sets in a given Hausdorff space,
then the respective operator $\Phi$ satisfying conditions (1)--(5) is called a Borel lifting. Note that von Neumann and Stone \cite{vNS} proved the existence of a lifting for a Borel measure space on $[0,1]$ under the assumption of the continuum hypothesis ({\bf CH}). A simple proof of the same result was then given by Musia{\l} \cite{Mu}. This was later generalized by Mokobodzki \cite{Mo} and Fremlin \cite{Fr} who showed that, subject to {\bf CH}, any $\sigma$-finite measure space with the measure algebra of cardinality $\leq \omega_2$ has a lifting. On the other hand, Shelah \cite{Sh} proved that it is consistent with ZFC that there exists no Borel lifting for Lebesgue measure on $[0,1]$.

We will focus on a particular case. We assume that $S$ is the $\sigma$-algebra $\B$ of Borel subsets of an uncountable Polish space $X$ and $J$ is the $\sigma$-ideal 
$[X]^{\leq \omega}$ of all countable subsets of $X$. Since any two uncountable Borel subsets of Polish spaces are Borel isomorphic
 \cite[Thm 3.3.13]{Sr}, it does not matter which Polish space is considered.

\begin{theorem}
For an uncountable Polish space $X$, the following conditions are equivalent:
\begin{itemize}
\item[(i)] {\bf CH};
\item[(ii)] there exists a lifting $\Phi\colon\B\to\B$ with respect to $[X]^{\leq \omega}$;
\item[(iii)] there exists a lower density operator $\Phi\colon \B\to\B$ with respect to  $[X]^{\leq \omega}$.
\end{itemize}
\end{theorem}
\begin{proof}
Implication (i)$\implies$(ii) follows from \cite[Thm 1]{Mu}. Implication (ii)$\implies$(iii) is obvious.

To prove (iii)$\implies$(i) assume $\neg${\bf CH}. We work with $X:=\R\times\R$. Enumerate $\R$ as $\{x_\alpha\colon\alpha<\mathfrak{c}\}$. Suppose that $\Phi\colon\mathcal{B}\to\mathcal{B}$ is a lower density operator with respect to $[X]^{\leq \omega}$. Let $Q_\alpha:=\Phi(P_\alpha)$ where $P_\alpha:=\{x_\alpha\}\times \R$ for $\alpha<\mathfrak{c}$. Note that if $\alpha\neq\beta$ then $Q_\alpha\cap Q_\beta=\Phi (P_\alpha\cap P_\beta)=\emptyset$ by (4) and (1). Let $\pi_2\colon \R\times \R\to \R$ be given by $\pi_2(x,y):=y$.

{\bf Claim.} \emph{There is $x\in \R$ such that $\{\alpha<\ce\colon x\in \pi_2[Q_\alpha]\}$ is uncountable.}

{\bf Proof of Claim.} Suppose that $\vert\{\alpha<\ce\colon x\in \pi_2[Q_\alpha]\}\vert\leq\omega$ for each $x\in \R$. Let $L_\alpha:=\{\beta<\ce\colon x_\alpha\in \pi_2[Q_\beta]\}$ for $\alpha<\ce$. Then $\vert\bigcup_{\alpha<\omega_1}L_\alpha\vert\leq\omega_1$ by our supposition. By $\neg${\bf CH}, the set $\ce\setminus\bigcup_{\alpha<\omega_1}L_\alpha$
is nonempty (of cardinality $\ce$).  Moreover, $\{x_\alpha\colon\alpha<\omega_1\}\cap \pi_2[Q_\xi]
=\emptyset$ for each $\xi\in\ce\setminus\bigcup_{\alpha<\omega_1}L_\alpha$. Thus $\{x_\alpha\colon\alpha<\omega_1\}\subseteq\R\setminus \pi_2[Q_\xi]=\pi_2[P_\xi]\setminus \pi_2[Q_\xi]\subseteq\pi_2[P_\xi\setminus Q_\xi]$ which gives a contradiction since $|\pi_2[P_\xi\setminus Q_\xi]|\leq |P_\xi\setminus\Phi(P_\xi)|\leq\omega$ by (3). 

Take $x\in \R$ as in the Claim. Consider the closed set $P:=\R\times\{x\}$. Then $|P\cap P_\alpha|=1$ and
$\Phi(P)\cap Q_\alpha=\Phi(P)\cap\Phi(P_\alpha)=\Phi(P\cap P_\alpha)=\emptyset$ for each $\alpha<\ce$, by (4), (2) and (1). Therefore 
$\Phi(P)\cap\bigcup_{\alpha <\ce}Q_\alpha=\emptyset$. On the other hand, $\vert P\cap\bigcup_{\alpha <\ce}Q_\alpha\vert>\omega$ by the choice of $x$. Thus $P\setminus\Phi(P)$ is uncountable and we reach a contradiction with (3). 
\end{proof}

Note that implication (iii)$\implies$(ii) follows from \cite[Theorem 2.8]{HLL}.

The above theorem answers a question posed by Jacek Hejduk during his invited talk given on the Conference on Real Function Theory
in Star\'a Lesn\'a in September 2016. He asked about the existence of a lower density operator on $\B(\R)$ with respect to $[\R]^{\leq \omega}$. Let us mention that lower density operators play an important role in  constructions of density like topologies; see \cite{W}, \cite{HLL}.

{\bf Acknowledgement.}
The first author would like to thank Kazimierz Musia{\l} and Jacek Hejduk for their useful comments.

\end{document}